\documentclass[10pt, a4paper]{amsart}
\usepackage[english]{babel}
\usepackage{amssymb,amsmath,amsthm,mathrsfs}
\usepackage{verbatim}
\usepackage[margin=1in]{geometry}
\usepackage{stmaryrd}

\newtheorem{theorem}{Theorem}
\newtheorem*{conjecture}{Conjecture}

\newtheorem{lemma}[theorem]{Lemma}
\newtheorem{remark}[theorem]{Remark}

\def\bal{\begin{aligned}}
\def\eal{\end{aligned}}
\def\be{\begin{equation}\label}
\def\ee{\end{equation}}
\def\bcs{\begin{cases}}
\def\ecs{\end{cases}}
\def\={\;=\;}
\def\+{\,+\,}
\def\-{\,-\,}

\def\Z{{\mathbb Z}}

\def\Q{{\mathbb Q}}
\def\R{{\mathbb R}}
\def\F{{\mathbb F}}
\def\P{{\mathbb P}}
\def\Fr{\mathcal{F}}

\def\supp{{\rm supp}}
\def\lb{\llbracket}
\def\rb{\rrbracket}
\def\ord{\mathrm{ord}}

\title{Higher Hasse--Witt matrices}
\author{Masha Vlasenko}
\thanks{This work was supported by the National Science Centre of Poland, grant UMO-2016/21/B/ST1/03084.}
\email{masha.vlasenko@gmail.com}
\address{Institute of Mathematics of the Polish Academy of Sciences \\ \'{S}niadeckich 8, 00-656 Warsaw, Poland}

\begin{document}
%\date{\today}
\maketitle

\begin{abstract} We prove a number of $p$-adic congruences for the coefficients of powers of a multivariate polynomial  $f(x)$ with coefficients in a ring $R$ of characteristic zero. If the Hasse--Witt operation is invertible, our congruences yield $p$-adic limit formulas which conjecturally describe the Gauss--Manin connection and the Frobenius operator on the unit-root crystal attached to $f(x)$. As a second application, we associate with $f(x)$ formal group laws over $R$. Under certain assumptions these formal group laws are coordinalizations of the Artin--Mazur functors.
\end{abstract}

%For a multivariate polynomial  $f(x)$ with coefficients in a ring $R$ we construct a sequence of matrices with entries in $R$ whose reductions modulo $p$ give iterates of the Hasse--Witt operation for the hypersurface of zeroes of the reduction of $f(x)$ modulo $p$.  We show that our matrices satisfy a system of congruences modulo powers of $p$. If the Hasse--Witt operation is invertible these congruences yield $p$-adic limit formulas, which conjecturally describe the Gauss--Manin connection and the Frobenius operator on the attached unit-root crystal attached to $f(x)$.

\section{Introduction and main results}\label{sec:intro}

Let $R$ be a commutative characteristic zero ring and $f \in R[x_1^{\pm 1},\ldots,x_n^{\pm 1}]$ be a Laurent polynomial in $n$ variables, which we will write as
\[
f (x) \= \sum_u \; a_u \, x^u\,, \quad a_u \in R\,,
\]
where the summation runs over a finite set of vectors $u \in \Z^n$. The \emph{support} of $f$ is the set of exponents of monomials in $f$, which we denote by $\supp(f) \= \{ u : a_u \ne 0\}$. The \emph{Newton polytope} $\Delta(f) \subset \R^n$ is the convex hull of $\supp(f)$. 

Consider the set of internal integral points of the Newton polytope 
\[
J \= \Delta(f)^o \,\cap\, \Z^n\,,
\]
where $\Delta(f)^o$ denotes the topological interior of $\Delta(f)$.  Let $g=\#J$ be the number of internal integral points in the Newton polytope, which we assume to be positive. Consider the following sequence of $g \times g$ matrices $\{ \beta_m; m \ge 0\}$ with entries in $R$ whose rows and columns are indexed by the elements of $J$:
\be{beta_mat}
( \beta_m )_{u,v \in J} \= \text{ the coefficient of } x^{m v - u} \text{ in } f(x)^{m-1}\,.
\ee
By convention, $\beta_1$ is the identity matrix. In this paper we show that matrices~\eqref{beta_mat} satisfy a number of $p$-adic congruences for each prime number $p$. 

\bigskip

Let us fix $p$ and restrict attention to the subsequence $\{ \alpha_s = \beta_{p^s}; s \ge 0\}$. The entries of these matrices are given by
\be{alpha_mat}
( \alpha_s )_{u,v \in J} \= \text{ the coefficient of } x^{p^s v - u} \text{ in } f(x)^{p^s-1}\,.
\ee
To state our first main result we recall that a \emph{Frobenius endomorphism} $\sigma: R \to R$ is a ring endomorphism which lifts the $p$th power endomorphism on $\bar R=R/pR$, that is 
\be{Frob}\sigma(r) \;\equiv\; r^p \mod p \quad \text{ for every } \quad r \in R\,.
\ee
For example, in the ring of integers $R=\Z$ the trivial endomorphism $\sigma(r)=r$ is a Frobenius endomorphism, and in the polynomial ring $R=\Z[t]$ one can take $(\sigma r)(t)= r(t^p)$. A \emph{derivation} $D : R \to R$ is an additive map satisfying $D(r_1 \cdot r_2) = D(r_1)r_2 + r_1 D(r_2)$. The ring of integers has no non-zero derivations, derivations of $\Z[t]$ are given by $D=r(t)\frac d{dt}$ with $r \in \Z[t]$. Below Frobenius endomorphisms and derivations of $R$ are applied to $g \times g$ matrices entry-wise. 

% Below we apply $\sigma$ to $g \times g$ matrices with entries in $R$ entry-wise. This is an endomorphism of the ring of matrices but not a $p$th power Frobenius endomorphism: property~\eqref{Frob} will not be satisfied in general for matrices of size $g > 1$. 
%In Sections~\ref{sec:lemma} and~\ref{sec:proof} we prove the following

\bigskip

\begin{theorem}\label{alpha_congs}
\begin{itemize}
\item[(i)] For every Frobenius endomorphism $\sigma: R \to R$ and every $s \ge 1$
\[
\alpha_s \;\equiv\;  \alpha_1 \cdot \sigma(\alpha_1) \cdot \ldots \cdot \sigma^{s-1}(\alpha_1)    \mod p \,.
\]

\item[(ii)] Assume $\bar \alpha_1 = \alpha_1 \mod p$ is invertible over $\bar R = R/pR$. Then 
\[
\alpha_{s+1} \cdot \sigma(\alpha_{s})^{-1} \;\equiv\; \alpha_{s} \cdot \sigma(\alpha_{s-1})^{-1}  \mod \; p^s \, 
\]
for all $s \ge 1$.
\item[(iii)] Under the condition of~(ii), for any derivation $D: R \to R$ one has
\[
D(\sigma^m(\alpha_{s+1})) \cdot \sigma^m(\alpha_{s
+1})^{-1} \;\equiv\; D(\sigma^m(\alpha_{s})) \cdot \sigma^m(\alpha_{s})^{-1} \mod \, p^{s+m} \, 
\]
for all $s,m \ge 0$. 
\end{itemize}
\end{theorem}

\bigskip

An important role in the theorem is played by the matrix 
\be{HW}
( \bar{\alpha}_1 )_{u,v \in J} \= \text{ the coefficient of } x^{p v - u} \text{ in } \bar{f}(x)^{p-1}\,,
\ee
the \emph{Hasse--Witt matrix} of the polynomial $\overline f \= f \mod p  \in \bar{R}[x_1^{\pm 1},\ldots,x_n^{\pm 1}]$. Denote by $\bar{\sigma}(r)=r^p$ the $p$th power endomorphism of $\bar{R}$.  If $\bar{R}=\F_q$, a finite field with $q=p^k$ elements, the characteristic polynomial
\be{char-pol-HW}
\det\Bigl(1 \- T \cdot \bar\alpha_1 \cdot \bar\sigma(\bar\alpha_1)\cdot \ldots \cdot\bar\sigma^{k-1}(\bar\alpha_1) \Bigr) \; \in \; \F_p[T]
\ee
is congruent modulo $p$ to a factor in the zeta function 
\[
\mathcal{Z}(X_{\bar f}/\F_q; T) \= \exp\Bigl( \sum_{m \ge 1} \# X_{\overline f}(\F_{q^m})\frac{T^m}{m} \Bigr)\,,
\]
which is well known to be a rational function of $T$. We should mention that~\eqref{HW} is a matrix of a $\bar\sigma$-linear operation (Hasse--Witt operation), which can be defined in geometric terms. The reader can find a detailed explanation of this for non-singular projective hypersurfaces in~\cite[\S 1]{MV16Creswick} and references there in. Similar results hold for toric hypersurfaces over finite fields. Here we only want to observe that due to $\bar\sigma$-linearity, the matrix of the $s$th iterate of the Hasse--Witt operation is given by $\bar\alpha_1 \cdot \bar\sigma(\bar\alpha_1) \cdot \ldots \cdot \bar\sigma^{s-1}(\bar\alpha_1)$, which is equal to $\alpha_s \mod p$ due to Theorem~\ref{alpha_congs}~(i). By this reason we called $\alpha_s$ \emph{higher Hasse--Witt matrices}. For a generic non-singular projective hypersurface over a finite field, the Hasse--Witt matrix~\eqref{HW} is known to be invertible (see~\cite{Mill76,Kob75} or \cite{AS16} for an elementary proof). 

%\bigskip
% In the case of one internal integral point in the Newton polytope ($g=1$) congruences~(i) and~(ii) were proved in~\cite{MV13} and in~\cite{SvS09}. In this case~(ii) is called Dwork's congruence and special cases of it were proved by Bernard Dwork in~\cite{Dw-p-adic-cycles}. However our exposition is independent of the methods used in these papers.  

\bigskip

When the Hasse--Witt matrix is invertible, Theorem~\ref{alpha_congs} implies existence of $p$-adic limits 
\be{lim_Frob}
F_{\sigma} \= \underset{s \to \infty} \lim \alpha_{s+1} \cdot \sigma(\alpha_{s})^{-1}
\ee
and
\be{lim_D}
N_D \= - \underset{s \to \infty} \lim D(\alpha_{s}) \cdot \alpha_{s}^{-1}\,.
\ee
These are $g \times g$ matrices with entries in the $p$-adic closure $\widehat R_p =\underset{\leftarrow}{\lim} R/p^s R$. Note that due to~(i) 
\[
F_{\sigma} \mod p \= \bar \alpha_1\,.
\]
Let us give an example. Consider $f(x)=1+ax+bx^2+x^3$ over $R=\Z[a,b]$. Matrices~\eqref{beta_mat} are given by
\[
\beta_m \= \text{ coefficients of } \begin{pmatrix} x^{m-1} & x^{2m-1} \\ x^{m-2} & x^{2m-2} \end{pmatrix} \text{ in } \bigl( 1 \+ ax \+ b x^2 \+ x^3 \bigr)^{m-1}\,,
\]
the first few being
\[
\beta_1\=\begin{pmatrix}1&0\\0&1\end{pmatrix}\,,\quad \beta_2\=\begin{pmatrix}a&1\\1&b\end{pmatrix}\,,\quad \beta_3\=\begin{pmatrix}a^2+2b&2b\\2a&2a+b^2\end{pmatrix}\,,\quad \ldots
\]
With the help of a computer one can easily check if $\bar\beta_p \mod p$ is invertible, which seems to be the case for $p \ne 3$. One has $\widehat R_p = \Z_p \{\{ a,b \}\}$, the ring of $p$-adically convergent power series in $a$ and $b$, and we recognise that
\be{ex-cubic}\bal
-\frac {d \beta_{p^s}}{da} \; \beta_{p^s}^{-1} \;\equiv\; \frac1{disc(f)} \begin{pmatrix}3 b + ab^2-4a^2 &  ab-9 \\2(3a-b^2)  & 2(3b-a^2)\end{pmatrix} \mod p^s\,, \\
-\frac {d \beta_{p^s}}{db} \; \beta_{p^s}^{-1} \;\equiv\; \frac1{disc(f)} \begin{pmatrix}2(3a-b^2) & 2(3b-a^2)\\ab-9  & 3 a + ba^2-4b^2\end{pmatrix}\mod p^s \,,\\
\eal\ee   
where 
\[
disc(f) \= -4(a^3+b^3) \+  a^2 b^2 \+ 18 \, ab \- 27
\]
is the discriminant of the cubic $f(x)$. What we see in the right-hand side of~\eqref{ex-cubic} are the matrices $N_{d/da}$ and $N_{d/db}$ respectively. We observe that they are independent of $p$ and have rational functions as their entries.

In fact, we can identify the matrices from~\eqref{ex-cubic} in the following way. A \emph{connection} $\nabla$ on an $R$-module $M$ is a rule that gives for every derivation $D: R \to R$ an additive map $\nabla_D: M \to M$ satisfying $\nabla_D(r m) = D(r) m + r \nabla_D (m)$ for any $r \in R$ and $m \in M$. A connection can be thought as a way to differentiate elements of $M$. There is a natural way to introduce a connection on the quotient
\[
H \= R[x] / \{f(x)=0 \} \= R 1 \oplus R x \oplus R x^2\,. 
\]
Namely, to differentiate in $a$ one thinks that the other variable is a constant and applies the usual rules of differential calculus to the equation $f(x)=0$:
\[
0 \= d f(x) \= f'(x) dx + x da \quad\Rightarrow\quad \frac{dx}{da} \= - \frac{x}{f'(x)}\,,
\]
where $f'(x)=a+2bx+3x^2$ is the derivative in $x$. Next, we need to do reduction modulo $f(x)=0$ to express the result of differentiation in the basis $1,x,x^2$. Using polynomials $A,B \in R[x]$ such that $A(x) f(x) + B(x) f'(x) = disc(f)$, one finds that
\[
\nabla_{\frac d{da}}(x^m) \= m x^{m-1} \frac{dx}{da} \= - \frac{m x^m}{f'(x)} \;\equiv\; - \frac{m x^m B(x)}{disc(f)} \;\equiv\; - \frac{m}{disc(f)} \cdot \Bigl(\; x^m B(x) \mod f(x) \;\Bigr)\,.
\]
This computation makes sense if we work over the ring $R=\Z[a,b,\frac1{disc(f)}]$, which we assume we do from now on. We find that 
\[\bal
&\nabla_{\frac d{da}}(x) \= \frac{dx}{da} \= - \frac{x}{f'(x)} \;\equiv\; \frac1{disc(f)} \Bigl((2 b^2-6a)+(ab^2-3b-2a^2)x + (ab-9)x^2 \Bigr)\,,\\
&\nabla_{\frac d{da}}(x^2) \= \frac{d(x^2)}{da} \= 2 x \frac{dx}{da} \= - \frac{2x^2}{f'(x)} \;\equiv\; \frac1{disc(f)} \Bigl(
(-2 ab+18)+(4b^2-2a^2b+6a)x \\
&\hskip10cm  + (12b-4a^2)x^2 \Bigr)\,,
\eal\]
and similarly for the other variable:
\[\bal
0 \= &d f(x) \= f'(x) dx + x^2 db \;\Rightarrow\;\\
&\nabla_{\frac d{db}}(x) \= \frac{dx}{db} \= - \frac{x^2}{f'(x)} \;\equiv\; \frac1{disc(f)} \Bigl((9-ab)+(2b^2-a^2b+3a)x + (6b-2a^2)x^2 \Bigr)\,,\\
&\nabla_{\frac d{db}}(x^2) \= \frac{d(x^2)}{db} \= 2 x \frac{dx}{db} \= - \frac{2x^3}{f'(x)} \;\equiv\; \frac1{disc(f)} \Bigl(
(4a^2-12b)+(-14ab+4a^3+18)x \\
& \hskip10cm + (-8b^2+2a^2b+6a)x^2 \Bigr)\,.
\eal\]
Element $1 \in H$ is \emph{horizontal}, that is  $\nabla_D(1)=0$ for any derivation $D$.  Connection rules then imply that the submodule $R 1 \subset H$ is preserved by $\nabla$, and therefore $\nabla$ descends to the rank~2 quotient 
\[
H' \= H/R1 \;\cong\; R x \oplus R x^2\,,
\]
where it is represented by the matrices 
\be{ex-cubic-1}\bal
(\nabla_{\frac d{da}}(x), \nabla_{\frac d{da}}(x^2)) &\= (x,x^2) \frac1{disc(f)} \begin{pmatrix} ab^2-3b-2a^2&4b^2-2a^2b+6a\\ab-9 & 12b-4a^2 \end{pmatrix}\,,\\
(\nabla_{\frac d{db}}(x), \nabla_{\frac d{db}}(x^2)) &\= (x,x^2) \frac1{disc(f)} \begin{pmatrix} 2b^2-a^2b+3a&-14ab+4a^3+18\\6b-2a^2 & -8b^2+2a^2b+6a \end{pmatrix}\,.
\eal\ee
One can easily check that in the basis $(x^2+b x, x)$ this connection is given precisely by the matrices in~\eqref{ex-cubic}. Connection rules imply that if $N_D$ is the matrix of $\nabla_D$ in some basis $(e_i)$, then its matrix $N_D'$ in another basis $(e_i')=(e_i)W$ is given by $N_D'= W^{-1}N_DW + W^{-1}D(W)$. In our case one applies this transformation with $W=\begin{pmatrix} 0&1\\1 & -b \end{pmatrix}$ to pass from~\eqref{ex-cubic} to~\eqref{ex-cubic-1}.

The situation appears more subtle with matrices $F_{\sigma}$. Firstly, we don't recognise their entries as rational functions. Secondly, these matrices depend on $p$. However one can give a rule to produce them all from an object independent of $p$ and $\sigma$.  We find experimentally that for $p\ne 3$
\be{ex-cubic-Frob}
\beta_{p^s} \cdot \sigma(\beta_{p^{s-1}})^{-1} \equiv U \cdot  F^0 \cdot  \sigma(U)^{-1} \mod p^s 
\ee
where 
\[
U(a,b) \= \begin{pmatrix}1+\frac29 ab + \ldots &-\frac13 a+\frac19b^2+\ldots \\
-\frac13 b+\frac19a^2+\ldots & 1+\frac29 ab + \ldots \end{pmatrix} \quad\in\quad Mat_{2 \times 2} \Bigl( \Z[\frac13]\lb a,b\rb \Bigr)
\] is the solution to 
\[
D(U)+N_D U \= 0 \text{ for all }D \in Der(R) \,,\qquad U\Big|_{a=b=0}=\begin{pmatrix}1&0\\0 & 1\end{pmatrix}
\]
and 
$F^0=\begin{pmatrix}1&0\\0 & 1\end{pmatrix}$ or $\begin{pmatrix}0&-1\\-1 & 0\end{pmatrix}$  when $\Bigl(\frac{-3}{p}\Bigr)=1$ or $-1$ respectively. Columns of $U$ represent horizontal elements for $\nabla$ in $H'$ in the basis $(x^2+bx,x)$. Thinking of $H'=H'_f$ as being attached to $f$, the $\widehat R_p$-linear map 
\[
F_\sigma : H'_{\sigma(f)}\otimes_R \widehat R_p \to H'_f \otimes_R \widehat R_p \,,\quad \sigma(f)=1+\sigma(a)x+\sigma(b)x^2+x^3 
\]
given in the bases $(x^2+\sigma(b) x, x)$ and $(x^2+bx,x)$ by the matrix $F_\sigma=U F^0 \,\sigma(U)^{-1}$  from~\eqref{ex-cubic-Frob} is a homomorphism of connections for any constant matrix $F^0$. The particular $F^0$ that occurs in~\eqref{ex-cubic-Frob} is such that the zeta function of $x^2-x+1=0$ over $\F_p$
\[
\mathcal{Z}(\{x^2-x+1=0\}/\F_p; T) \= \bcs \exp\Bigl( \sum_{m \ge 1} \frac{2}{m} T^m \Bigr)\= \frac1{(1-T)^2}\,, \quad \Bigl(\frac{-3}{p}\Bigr)=1 \,, \\ \exp\Bigl( \sum_{m \ge 1} \frac{2}{2m} T^{2m} \Bigr)\= \frac1{1-T^2}\,, \quad  \Bigl(\frac{-3}{p}\Bigr)=-1 \ecs
\]
is equal to $\det(1-T\cdot F^0)^{-1}$. This zeta function can be thought as the dependent on $p$ factor in the zeta function of $x^3+1=0$, the specialisation of the equation $f(x)=0$ at $a=b=0$.      

\bigskip

Though the results and methods in this paper are essentially elementary, we will allow ourselves to state the following conjecture regarding the meaning of the limiting matrices $F_{\sigma}$ and $N_D$ in the general case. 

\begin{conjecture}
When the Hasse--Witt matrix~\eqref{HW} is invertible,  matrices~\eqref{lim_Frob} and~\eqref{lim_D} describe the Frobenius operator and the Gauss--Manin connection respectively on the unit-root crystal attached to $f$.   
\end{conjecture}

What this conjecture means for $F_{\sigma}$ can be seen in elementary terms when $R=\Z_q$, $q=p^k$. As we mentioned earlier, \eqref{char-pol-HW} is congruent modulo $p$ to a factor in the zeta function $\mathcal{Z}(X_{\bar f}/\F_q; T)$. Now we claim this congruence lifts to
\be{zeta-factor-lift}
\det\Bigl(1 \- T \cdot F_{\sigma} \cdot \sigma(F_{\sigma}) \cdot \ldots \cdot \sigma^{k-1}(F_{\sigma}) \Bigr) \; \in \; \Z_p[T]
\ee
being an actual factor of $\mathcal{Z}(X_{\bar f}/\F_q; T)$. The reader will find some evidence for this claim in Section~\ref{sec:ASD}. Note that by construction the reciprocal roots of~\eqref{zeta-factor-lift} are $p$-adic units. We further claim that they are actually all non-trivial $p$-adic unit-roots. By non-trivial we mean that they are eigenvalues of the Frobenius action on the \emph{primitive} cohomology of the hypersurface $X_{\bar f}$.

To explain the statement of the conjecture in general, let $X_f=\{ f = 0 \}$ be the hypersurface of zeroes of $f$. When  $X_f$ is non-singular its de Rham cohomology $H_f=H^{\cdot}_{dR}(X_f)$ is an $R$-module equipped with the following structure. First, there is the Gauss--Manin connection $\nabla$ on $H_f$. This is the natural connection we considered in the above example, where the hypersurface was of dimension zero. In addition, for every Frobenius endomorphism $\sigma \in End(R)$ there is a homomorphism of connections $F_{\sigma}:H_{\sigma(f)}\otimes_R \widehat R_p \to H_f \otimes_R \widehat R_p$. This map comes via the comparison isomorphism of $H_f \otimes_R \widehat R_p$ and $p$-adic cohomology of $X_{\bar f}$. An $R$-module with such a structure is called a \emph{crystal} (see~\cite{Katz85}). Our conjecture claims that when $\det \bar\alpha_1$ is invertible then there is a rank $g$ free subquotient in $H_f\otimes_R \widehat R_p$ where both Frobenius structure and connection descend, and their matrices in a certain basis are given by~\eqref{lim_Frob} and~\eqref{lim_D} respectively. This subquotient can be characterized as the largest part of the primitive cohomology, where the Frobenius action is invertible.

We plan to write a paper devoted to the proof of the above conjecture for toric hypersurfaces by methods of Dwork cohomology. An explicit construction of the unit-root crystal for a projective hypersurface was given in~\cite{Katz85}. In fact, the main result~\cite[Theorem 6.2]{Katz85} states a transposed version of congruences~(ii) and~(iii) in Theorem~\ref{alpha_congs} for matrices composed of certain expansion coefficients of differential forms on $X_f$. Parallelism between our results and Katz's work was one of the motivations for stating the conjecture. We were very pleased to learn that recently An Huang, Bong Lian, Shing-Tung Yau and Chenglong Yu proved our conjecture by Katz's expansion method, providing the desired link between the two approaches, see \cite{HLYY18}.

\bigskip  

Our second main result states integrality of certain formal group laws associated to a Laurent polynomial $f \in R[x_1^{\pm 1},\ldots,x_n^{\pm 1}]$. Let $J$ be either the set $\Delta(f) \cap \Z^n$ of all integral points in the Newton polytope of $f$ or the subset of internal integral points $ \Delta(f)^\circ \cap \Z^n$ as above. Assume that $J$ is non-empty and let $g = \#J$. Consider the sequence of matrices $\beta_m \in {\rm Mat}_{g \times g}(R)$, $m \ge 1$ given by~\eqref{beta_mat} and define a $g$-tuple of formal powers series $l(\tau) = (l_u(\tau))_{u \in J}$ in $g$ variables $\tau=(\tau_v)_{v \in J}$ by formula
\[
l(\tau) \= \sum_{m=1}^{\infty} \frac1m \, \beta_{m} \, \tau^m\,.
\]   
Consider the $g$-dimensional formal group law 
\be{fgl-f}
G_f(\tau,\tau') = l^{-1}(l(\tau)+l(\tau'))
\ee
with coefficients in $R \otimes \Q$. (The reader unfamiliar with formal group laws can simply think of $G_f(\tau,\tau')$ as a formal power series in $2g$ variables.) Recall that $R$ is a characteristic zero ring, that is the natural map $R \to R \otimes \Q$ is an embedding.  Let $p$ be a prime number and $R_{(p)} = R \otimes \Z_{(p)} \subset R \otimes \Q$ be the subring formed by elements without $p$ in the denominator. Then

\begin{theorem}\label{FGL_theorem} If $R$ can be endowed with a $p$th power Frobenius endomorphism~\eqref{Frob}, then $G_f \in R_{(p)}\lb \tau,\tau'\rb$.
\end{theorem}

Note that if $R$ (or perhaps, a larger ring) admits a $p$th power Frobenius endomorphism for every prime $p$ then Theorem~\ref{FGL_theorem} implies that $G_f \in R\lb \tau, \tau' \rb$ because the subring $\cap_p R_{(p)} \subset R \otimes \Q$ coincides with $R$. As it was mentioned earlier, $R=\Z$ and $R\=\Z[t]$ are examples of rings admitting a Frobenius endomorphism for every $p$.

In certain cases (e.g. when $f$ is a homogeneous polynomial of degree $d>n$, the projective hypersurface $X_f$ is non-singular and $J=\Delta(f)^\circ \cap \Z^n$) formal group law~\eqref{fgl-f} is known to be   a coordinalisation of the Artin--Mazur formal group, see~\cite{St87}. Integrality and relation to $p$-adic cohomology then follow from geometry, but in contrast to~\cite{St87} the methods in this paper are elementary. Thus our theorem has no assumptions on $f$ and gives new integral formal group laws, whose geometric role yet has to be investigated.   

\bigskip

The paper is organized as follows. In Section~\ref{sec:lemma} we prove a number of lemmas on $p$-adic congruences for the powers of a Laurent polynomial. They are then used in Sections~\ref{sec:proof} and~\ref{sec:FGL} to prove Theorems~\ref{alpha_congs} and~\ref{FGL_theorem} respectively. In Section~\ref{sec:ASD} we give some evidence to support the above conjecture.

\bigskip

The results in this paper were earlier announced in~\cite{MV16Creswick}.

\bigskip

{\bf Acknowledgements.} I would like to thank John Voight and Susanne M\"{u}ller for their remarks and questions on the earlier versions of this manuscript. I grateful to Jan Stienstra for numerous inspiring discussions and to Alan Adolphson and Steven Sperber for our correspondence which motivated essential improvements in this paper. At the end, I would like to mention that my recent collaboration with Frits Beukers finally brought an understanding of the cohomological origin for the congruences discussed here. I hope to write a text about it in the nearest future.    

\bigskip

\section{Lemmas on congruences for the powers of $f(x)$}\label{sec:lemma} 

Let $R$ be a commutative ring with~$1$ and $p$ be a prime number. Assume that $R$ is endowed with a $p$th power Frobenius endomorphism, that is we have a ring endomorphism $\sigma: R \to R$ such that $\sigma(a) \;\equiv\; a^p \mod p R$ for every $a \in R$.

\begin{lemma}\label{delta_lemma} For $a \in R$ we define a sequence of elements $\delta_s = \delta_s(a) \in R$, $s = 1,2,\ldots$ by the recursive formula
\[
a^{p^s-1} \= \delta_1(a) \cdot \sigma(a^{p^{s-1}-1}) \+ \delta_2(a) \cdot \sigma^2(a^{p^{s-2}-1}) \+ \ldots \+ \delta_s(a) \,.  
\]
Then
\begin{itemize}
\item[(i)] $\delta_{s}(a) \in p^{s-1} R$ for every $s \ge 1$;
\item[(ii)] $a^{m p^s -1} \- \sum_{i=1}^s \delta_i(a) \cdot \sigma^i(a^{m p^{s-i}-1}) \in p^s R$ for every $m, s \ge 1$.
\end{itemize}
\end{lemma}

\begin{proof} For $s=1$ we have $\delta_1(a)=a^{p-1}$ and~(i) holds trivially. For higher $s$ we prove~(i) by induction. We have
\[\bal
\delta_s \= a^{p^s-1} \- \sum_{i=1}^{s-1} \delta_i \cdot \sigma^i(a^{p^{s-i}-1}) &\= a^{p^s-1} \- \sum_{i=1}^{s-1} \delta_i \cdot \sigma^i(a^{p^{s-1-i}-1}) \cdot \sigma^i(a^{p-1})^{p^{s-1-i}}\\
&\;\equiv\; a^{p^s-1} \- \sum_{i=1}^{s-1} \delta_i \cdot \sigma^i(a^{p^{s-1-i}-1}) \cdot a^{p^s-p^{s-1}} \mod p^{s-1} R\\
&\= \Bigl( a^{p^{s-1}-1} \- \sum_{i=1}^{s-1} \delta_i \cdot \sigma^i(a^{p^{s-1-i}-1}) \Bigr) \cdot a^{p^s-p^{s-1}} \= 0\,,
\eal\]
where the congruence in the middle row above holds for every $i$ for the following reason. We notice that $\sigma^i(a^{p-1}) \equiv a^{p^i(p-1)} \mod pR$, then we raise this congruence $s-1-i$ times to the power $p$ and get $\sigma^i(a^{p-1})^{p^{s-1-i}} \equiv a^{p^{s-1}(p-1)} \mod p^{s-i} R$. By inductional assumption $\delta_i \in p^{i-1}R$, and hence
\[
\delta_i \cdot \sigma^i(a^{p-1})^{p^{s-1-i}} \equiv \delta_i \cdot a^{p^s-p^{s-1})} \mod p^{s-1} R\,.
\] 

We prove~(ii) in a similar manner:
\[\bal
a^{m p^s -1} &\- \sum_{i=1}^s \delta_i(a) \cdot \sigma^i(a^{m p^{s-i}-1}) \= a^{m p^s -1} \- \sum_{i=1}^{s-1} \delta_i(a) \cdot \sigma^i(a^{m p^{s-i}-1}) \- \delta_s(a) \cdot \sigma^s(a^{m-1})\\
& \qquad\qquad\qquad\qquad\qquad\quad  \overset{(p^s)}\equiv\; a^{m p^s -1} \- \sum_{i=1}^{s-1} \delta_i(a) \cdot \sigma^i(a^{m p^{s-i}-1}) \- \delta_s(a) \cdot a^{(m-1)p^s} \\
(\quad\text{because }\; &\sigma^s(a^{m-1}) \equiv a^{(m-1)p^s} \mod pR \quad\text{ and }\; \delta_s(a) \in p^{s-1} R \text{ by (i)}\quad)\\ 
&\= a^{m p^s -1} \- \sum_{i=1}^{s-1} \delta_i(a) \cdot \sigma^i(a^{m p^{s-i}-1}) \- \Bigl(a^{p^s-1} \- \sum_{i=1}^{s-1} \delta_i(a) \cdot \sigma^i(a^{p^{s-i}-1})\Bigr) \cdot a^{(m-1)p^s} \\
&\= \sum_{i=1}^{s-1} \delta_i(a) \cdot \Bigl( \sigma^i(a^{p^{s-i}-1}) \cdot a^{(m-1)p^s} \- \sigma^i(a^{m p^{s-i}-1})\Bigr)\\ 
&\overset{(p^{s})}\equiv\;\sum_{i=1}^{s-1} \delta_i(a) \cdot \Bigl( \sigma^i(a^{p^{s-i}-1}) \cdot \sigma^i(a^{m-1})^{p^{s-i}} \- \sigma^i(a^{m p^{s-i}-1})\Bigr)\\
(\quad \sigma^i(a^{m-1}) & \equiv a^{(m-1)p^i} \mod pR \;\Rightarrow\; \sigma^i(a^{m-1})^{p^{s-i}} \equiv a^{(m-1)p^s} \mod p^{s-i+1}R \;\Rightarrow\; \\
& \delta_i(a) \cdot \sigma^i(a^{m-1})^{p^{s-i}} \equiv \delta_i(a) \cdot a^{(m-1)p^s} \mod p^{s}R \quad\text{  since }\; \delta_i(a) \in p^{i-1} R \quad) \\
&\= \sum_{i=1}^{s-1} \delta_i(a) \cdot \Bigl( \sigma^i(a^{p^{s-i}-1} \cdot a^{p^{s-i}(m-1)}) \- \sigma^i(a^{m p^{s-i}-1})\Bigr) \=0\,.
\eal\]   
\end{proof}

\bigskip

Let $R' = R[x_1^{\pm 1},\ldots,x_n^{\pm 1}]$ be the ring of Laurent polynomials in $n$ variables. We extend $\sigma$ to a Frobenius endomorphism of $R'$ by assigning $\sigma(x_i)=x_i^p$ for $1 \le i \le n$. For $f \in R'$ we denote the Newton polytope of $f$ by $\Delta(f) \subset \R^n$. Lemma~\ref{delta_lemma} can be applied in the ring $R'$ and we obtain a sequence of Laurent polynomials $\{ \delta_s(f) ; s \ge 1\}$. The following lemma gives an estimate for their Newton polytopes. 

\begin{lemma}\label{Newt_polytope_lemma} $\Delta(\delta_s(f)) \subseteq (p^s-1) \Delta(f)$ for $s \ge 1$.
\end{lemma}
\begin{proof}
The statement obviously holds for $s=1$ since $\delta_1(f)=f^{p-1}$. For higher $s$ it follows by induction because each term in the sum on the right in
\[
\delta_s(f) \= f^{p^s-1} \- \sum_{i=1}^{s-1} \delta_i(f) \cdot \sigma^i(f^{p^{s-i}-1})
\] 
has its Newton polytope inside $(p^s-1)\Delta(f)$. Indeed, this is obvious for $f^{p^s-1}$ and for each $1 \le i \le s-1$ one has $\Delta(\sigma^i(f^{p^{s-i}-1})) \subseteq p^i \cdot(p^{s-i}-1) \Delta(f)$ and 
\[
\Delta(\delta_i(f) \cdot \sigma^i(f^{p^{s-i}-1})) \subseteq \Bigl((p^i-1) \+ p^i \cdot(p^{s-i}-1) \Bigr)\Delta(f) \= (p^s-1) \Delta(f)\,. 
\] 
\end{proof}

We assume that the set of internal integral points $J = \Delta(f)^\circ \cap \Z^n$ is non-empty and let $g= \# J$. The endomorphism $\sigma \in End(R)$ naturally extends to an endomorphism of the ring of $g \times g$ matrices with entries in $R$  (we simply apply it to each matrix entry). We will denote this extension by the same letter $\sigma \in End\bigl({\rm Mat}_{g \times g}(R)\bigr)$. However it is not a Frobenius endomorphism any more: the property $\sigma (\alpha) \equiv \alpha^p \mod p$ will not hold in general for $\alpha \in  {\rm Mat}_{g \times g}(R)$ when $g > 1$.  

Recall that for $m \ge 1$ matrices 
\[
( \beta_m )_{u,v \in J} \= \text{ the coefficient of } x^{m v - u} \text{ in } f(x)^{m-1}
\]
were defined in~\eqref{beta_mat}. By convention, $\beta_1$ is the identity matrix. We also use the notation $\alpha_s \= \beta_{p^s}$, $s \ge 0$ as in~\eqref{alpha_mat}. 

\begin{lemma}\label{mat_cong_lemma} For $s \ge 1$ consider $g \times g$ matrices given by
\be{gamma_mat}
( \gamma_s )_{u,v \in J} \= \text{ the coefficient of } x^{p^s \, v - u} \text{ in } \delta_s(f)\,.
\ee
We have
\begin{itemize}
\item[(i)] $\gamma_s \in p^{s-1} {\rm Mat}_{g \times g}(R)$ for $s \ge 1$;
\item[(ii)] $\alpha_s \= \gamma_1 \cdot \sigma(\alpha_{s-1}) \+ \gamma_2 \cdot \sigma^2(\alpha_{s-2}) \+ \ldots \+ \gamma_{s-1} \cdot \sigma^{s-1}(\alpha_1) \+ \gamma_s$ for $s \ge 1$;
\item[(iii)] for $m,s \ge 1$
\[
\beta_{mp^s} \- \sum_{i=1}^s \gamma_i \cdot \sigma^i(\beta_{m p^{s-i}}) \quad \in \quad p^{s} {\rm Mat}_{g \times g}(R)\,.
\]
\end{itemize}
\end{lemma}
\begin{proof} (i) is clear since $\delta^s(f) \in p^{s-1} R'$ by~(i) in Lemma~\ref{delta_lemma}. To prove~(ii) consider the identity
\be{id1}
f^{p^s-1} \= \sum_{i=0}^{s} \delta_i(f) \cdot \sigma^i(f^{p^{s-i}-1})\,. 
\ee
Let $u,v \in J$. In order to compute the coefficient of $x^{p^s v - u}$ in $\delta_i(f) \cdot \sigma^i(f^{p^{s-i}-1})$ we are interested in pairs of vectors $w \in \supp(\delta_i(f))$ and $\tau \in \supp(f^{p^{s-i}-1})$ such that
\be{id2}
w + p^i \tau \= p^s v \- u\,.
\ee
By Lemma~\ref{Newt_polytope_lemma} we have $w \in (p^i-1) \,\Delta(f)$, and since $u \in \Delta(f)^\circ$ it follows that $w+u \in p^i \,\Delta(f)^\circ$. Moreover,~\eqref{id2} implies that $w + u \in p^i \Z^n$, and therefore
$\frac{1}{p^i} (w + u) \in \Delta(f)^o \cap \Z^n \= J$. On the other hand, for every $\mu \in J$ vectors
\[
w \= p^i\mu \- u \,,\qquad \tau \= p^{s-i} v \- \mu  
\]
satisfy~\eqref{id2}. It follows that the coefficient of $x^{p^s v - u}$ in $\delta_i(f) \cdot \sigma^i(f^{p^{s-i}-1})$ is equal to $\sum_{\mu \in J} (\gamma_i)_{u,\mu} \, \sigma^i (\alpha_{s-i})_{\mu, v}$, and~(ii) now follows from~\eqref{id1}.

We prove~(iii) in a similar vein. By~(ii) in Lemma~\ref{delta_lemma} we have
\be{id3}
f^{m p^s-1} \- \sum_{i=0}^{s} \delta_i(f) \cdot \sigma^i(f^{m p^{s-i}-1}) \quad \in p^s R'\,. 
\ee
In order to compute the coefficient of $x^{m p^s v - u}$ in $\delta_i(f) \cdot \sigma^i(f^{m p^{s-i}-1})$ we look at pairs of vectors $w \in \supp(\delta_i(f))$ and $\tau \in \supp(f^{m p^{s-i}-1})$ such that $
w + p^i \tau \= m p^s v \- u$. The same argument as above shows that $\mu = (w+u)/p^i \in \Delta(f)^\circ \cap \Z^n \= J$. With this $\mu$ we can rewrite $w = p^i \mu - u$, $\tau = m p^{s-i} v - \mu$, and~(iii) thus follows from~\eqref{id3}.
\end{proof}

\begin{remark}\label{bigger_mat_remark} It is clear from the proof of Lemma~\ref{mat_cong_lemma} that we could use a larger set $\widetilde{J} = \Delta(f) \cap \Z^n$ of all integral points in the Newton polytope of $f$ instead of the set of internal integral points $J  = \Delta(f)^\circ \cap \Z^n$. The statement of Lemma~\ref{mat_cong_lemma} then holds for the sequences of larger matrices $\{ \widetilde{\beta}_m, m \ge 1\}$, $\{ \widetilde{\alpha}_s, s \ge 0\}$ and $\{ \widetilde{\gamma}_s , s \ge 1\}$ defined by formulas~\eqref{beta_mat}, \eqref{alpha_mat} and~\eqref{gamma_mat} respectively with $\widetilde{J}$ instead of $J$.  
\end{remark}

\section{Proof of Theorem 1}\label{sec:proof}

Our main tools for the proof will be Lemma~\ref{mat_cong_lemma}, parts~(i) and~(ii). We will also need the following observation relating derivations and Frobenius endomorphisms. 

\begin{lemma}\label{D_Phi_cong} For any derivation $D: R \to R$ and any Frobenius endomorphism $\sigma:R\to R$ one has $D (\sigma^m(a)) \;\in\; p^m R$ for all $a \in R$ and $m \ge 1$.
\end{lemma}
\begin{proof} Since $\sigma(a) = a^p + p b$ for some $b \in R$, then
\[
D(\sigma(a)) \= D(a^p) \+ p D(b) \= p (a^{p-1} D(a) \+ D(b)) \in p R\,,  
\]
which proves the statement for $m=1$. We will do induction on $m$. If the statement holds for $m-1$ then
\[\bal
D(\sigma^m(a)) \= D(\sigma^{m-1}(a^p + p b)) \= D\Bigl(\sigma^{m-1}(a)^p\Bigr) \+ p D(\sigma^{m-1}(b)) \\
\= p \, \sigma^{m-1}(a)^{p-1} D(\sigma^{m-1}(r)) \+ p \, D(\sigma^{m-1}(b)) \in p^m R
\eal\]
since both $D(\sigma^{m-1}(a))$ and $D(\sigma^{m-1}(b))$ belong to $p^{m-1}R$.
\end{proof}

\begin{proof}[Proof of Theorem~\ref{alpha_congs}] By (i)--(ii) in Lemma~\ref{mat_cong_lemma} we have $\alpha_s \equiv \gamma_1 \cdot \sigma(\alpha_{s-1}) \mod p$. Iteration yields
\[
\alpha_s \;\equiv\; \gamma_1 \cdot \sigma(\gamma_1) \cdot \ldots \cdot \sigma^{s-1}(\gamma_1) \mod p \,,
\] 
and part (i) follows immediately since $\alpha_1 = \gamma_1$.

We will prove (ii) by induction on $s$. We shall show that 
\[
\alpha_{s+1} \cdot  \sigma(\alpha_s)^{-1} \;\;\equiv\;\;\alpha_{s} \cdot \sigma(\alpha_{s-1})^{-1} \; \mod \; p^s\,.
\]
The case $s=1$ follows from part (i). Let us substitute the recursive expressions for $\alpha_s$ and $\alpha_{s+1}$ from~(ii) in Lemma~\ref{mat_cong_lemma} into the two sides of the desired congruence: 
\[\bal
&\alpha_{s+1} \cdot  \sigma(\alpha_s)^{-1} = \gamma_1 + \sum_{j=2}^{s+1} \gamma_j \cdot \sigma^j(\alpha_{s+1-j}) \cdot \sigma(\alpha_s)^{-1} \,,\\
&\alpha_{s} \cdot \sigma(\alpha_{s-1})^{-1} = \gamma_1 + \sum_{j=2}^{s} \gamma_j \cdot \sigma^j(\alpha_{s-j}) \cdot \sigma(\alpha_{s-1})^{-1}\,. \\
\eal\]
Since we want to compare these two expressions modulo $p^s$ and $\gamma_{s+1} \;\equiv\; 0 \mod p^s$, the last term in the upper sum can be ignored. For every $j=2,\ldots,s$ we use the inductional assumption as follows:
\[\bal
& \alpha_s \sigma(\alpha_{s-1})^{-1} \;\equiv\; \alpha_{s-1} \sigma(\alpha_{s-2})^{-1} \mod p^{s-1}\\
& \alpha_{s-1} \sigma(\alpha_{s-2})^{-1} \;\equiv\; \alpha_{s-2} \sigma(\alpha_{s-3})^{-1} \mod p^{s-2}\\
& \vdots \\
& \alpha_{s+2-j} \sigma(\alpha_{s+1-j})^{-1} \;\equiv\; \alpha_{s+1-j} \sigma(\alpha_{s-j})^{-1} \mod p^{s+1-j}\\
\eal\]
We then apply the respective power of $\sigma$ to each row and multiply these congruences out to get that modulo $p^{s+1-j}$
\[\bal
\alpha_s \sigma^{j-1}(\alpha_{s+1-j})^{-1} & \= \alpha_s \sigma(\alpha_{s-1})^{-1} \sigma(\alpha_{s-1}) \sigma^2(\alpha_{s-2})^{-1} \ldots \sigma^{j-1}(\alpha_{s+1-j})^{-1}\\
& \;\equiv\; \alpha_{s-1} \sigma(\alpha_{s-2})^{-1} \sigma(\alpha_{s-2}) \sigma^2(\alpha_{s-3})^{-1} \ldots \sigma^{j-1}(\alpha_{s-j})^{-1} \= \alpha_{s-1} \sigma^{j-1}(\alpha_{s-j})^{-1}\,.
\eal\]
By our assumption, determinants of these matrices are units in $\widehat R$. Hence we can invert them to get
\be{(ii)gen}
\sigma^{j-1}(\alpha_{s+1-j}) \alpha_s^{-1} \;\equiv\; \sigma^{j-1}(\alpha_{s-j})  \alpha_{s-1}^{-1} \mod p^{s+1-j}\,. 
\ee
Now we allpy $\sigma$ and multiply by $\gamma_j$. Since $\gamma_j \;\equiv\; 0 \mod p^{j-1}$ we get
\[
\gamma_j \, \sigma^j(\alpha_{s+1-j}) \sigma(\alpha_s)^{-1} \;\equiv\; \gamma_j \, \sigma^{j}(\alpha_{s-j})  \sigma(\alpha_{s-1})^{-1} \mod p^{s}\,. 
\]
Summation in $j$ gives the desired result~(ii).

For (iii) we shall show that
\[
D(\sigma^m(\alpha_{s+1})) \cdot \sigma^m(\alpha_{s+1})^{-1} \;\equiv\; D(\sigma^m(\alpha_{s})) \cdot \sigma^m(\alpha_{s})^{-1} \mod \, p^{s+m} \,
\]
for every $s,m\ge 0$. For $s=0$ the statement is true with any $m$: the right-hand side vanishes since $D(1)=0$  and the entries of the left-hand side belong to $p^m R$ by Lemma~\ref{D_Phi_cong}. We will now do induction on $s$. Substituting  the recursive expressions for $\alpha_s$ and $\alpha_{s+1}$ from~(ii) in Lemma~\ref{mat_cong_lemma} we can write 
\be{LHS_RHS_reps}\bal
D(\sigma^m(\alpha_s)) \sigma^m(\alpha_s)^{-1} &\= \sum_{i=1}^s D(\sigma^m(\gamma_i)) \, \sigma^m\Bigl( \sigma^{i}(\alpha_{s-i}) \, \alpha_s^{-1}\Bigr) \\
& \quad\quad \+ \sum_{i=1}^s \sigma^m(\gamma_i) \, D(\sigma^{m+i}(\alpha_{s-i})) \, \sigma^m(\alpha_s)^{-1}\,, \\
D(\sigma^m(\alpha_{s+1})) \sigma^m(\alpha_{s+1})^{-1} &\= \sum_{i=1}^{s+1} D(\sigma^m(\gamma_i)) \, \sigma^m\Bigl( \sigma^{i}(\alpha_{s+1-i}) \, \alpha_{s+1}^{-1}\Bigr) \\ 
& \quad\quad \+ \sum_{i=1}^{s+1} \sigma^m(\gamma_i) \, D(\sigma^{m+i}(\alpha_{s+1-i})) \, \sigma^m(\alpha_{s+1})^{-1}\,. 
\eal\ee 
Consider the terms with $i=s+1$ in the latter identity. In the first sum this term vanishes modulo $p^{s+m}$ because the entries of $\gamma_{s+1}$ are in $p^s R$, thus Lemma~\ref{D_Phi_cong} implies that the entries of $D(\sigma^m(\gamma_{s+1}))$ belong to $p^{s+m}R$. In the second sum this term vanishes since $D(\sigma^{m+s+1}(\alpha_0))=D(\alpha_0)=0$. 
Now take any $1 \le i \le s$. The respective terms in the first sums of both identities are equal modulo $p^{s+m}$ because
\[
D(\sigma^m(\gamma_i)) \equiv 0 \mod p^{m+i-1}
\]
by Lemmas~\ref{mat_cong_lemma}(i) and~\ref{D_Phi_cong} and
\[
\sigma^{i}(\alpha_{s-i}) \, \alpha_s^{-1} \equiv \sigma^{i}(\alpha_{s+1-i}) \, \alpha_{s+1}^{-1} \mod p^{s+1-i}
\]
as a consequence of part~(ii) of this theorem (e.g. take~\eqref{(ii)gen} with $i+1$ and $s+1$ instead of $j$ and $s$ respectively). It remains to compare the terms with index $i$ in the second sums of the two identities in~\eqref{LHS_RHS_reps}. We have
\[\bal
& D(\sigma^{m+i}(\alpha_{s+1-i})) \sigma^{m+i}(\alpha_{s+1-i})^{-1} \equiv D(\sigma^{m+i}(\alpha_{s-i})) \sigma^{m+i}(\alpha_{s-i})^{-1} \mod p^{s+m} \\
& \qquad\qquad\qquad\qquad\qquad\qquad (\text{ both } \; \equiv \; 0 \mod p^{m+i} \; \text{ by Lemma~\ref{D_Phi_cong}})\,, \\
& \sigma^{m+i}(\alpha_{s+1-i}) \sigma^m(\alpha_{s+1})^{-1} \equiv \sigma^{m+i}(\alpha_{s-i}) \sigma^m(\alpha_{s})^{-1} \mod p^{s+1-i}\,.\\
\eal\]
Here the first congruence follows from the inductional assumption and the last one follows from part~(ii) of this theorem. Myltiplying the above congruences we obtain
\[
D(\sigma^{m+i}(\alpha_{s+1-i})) \, \sigma^m(\alpha_{s+1})^{-1} \equiv D(\sigma^{m+i}(\alpha_{s-i})) \, \sigma^m(\alpha_{s})^{-1} \mod p^{s+m}\,.
\]
Multiplying both sides by $\sigma^{m}(\gamma_i)$ we see that the respective terms in~\eqref{LHS_RHS_reps} are congruent modulo $p^{s+m+i-1}$ (which is even better than we need whenever $i>1$). This acomplishes the proof of the inductional step.
\end{proof}

\section{Integrality of formal group laws attached to a Laurent polynomial}\label{sec:FGL}

In this section we prove Theorem~\ref{FGL_theorem}. The proof is based on Hazewinkel's functional equation lemma (\cite[\S 10.2]{Ha78}), the conditions of which are satisfied due to~(iii) in Lemma~\ref{mat_cong_lemma}.

\begin{proof}[Proof of Theorem~\ref{FGL_theorem}]
Recall that $R$ is a characteristic zero ring, that is the natural map $R \to R \otimes \Q$ is injective. We assume there is a $p$th power Frobenius endomorphism $\sigma: R \to R$, which we extend to $R \otimes \Q$ by linearity. We consider the case $J = \Delta(f)^\circ \cap \Z^n$ first. Let $g=\#J$ and $\{ \gamma_s; s \ge 1\}$ is the sequence of $g \times g$ matrices from Lemma~\ref{mat_cong_lemma}. Put $\mu_s = \frac1{p^{s-1}} \gamma_s$. Then $\mu_s \in {\rm Mat}_{g \times g}(R)$ by~(i) in Lemma~\ref{mat_cong_lemma}. We shall now check that each power series in the tuple 
\[
h(\tau) \= l(\tau) \- \frac1p \, \sum_{s=1}^{\infty} \mu_s \, (\sigma^s l)(\tau)  
\]   
has coefficients in $R_{(p)}$. Here $\sigma$ extends to $(R \otimes \Q) \lb \tau \rb$ by assigning $\sigma(\tau_u) = \tau_u^p$ for each $u \in J$, and it then acts on tuples of power series coordinate-wise. For any $k \ge 1$ we write $k = m p^r$ where $(m,p)=1$. Then
\[\bal
\text{ the coefficient of } \tau_v^k \text{ in } h_u(\tau) &\= \frac1{k} (\beta_{k})_{u,v} \- \frac1p \sum_{s=1}^{r} \sum_{w \in J} (\mu_s)_{u,w} \frac1{mp^{r-s}} (\beta_{m p^{r-s}})_{w,v}\\
&\= \frac1{m p^r} \Bigl( (\beta_{mp^r})_{u,v} \- \sum_{s=1}^{r} \sum_{w \in J} (\gamma_s)_{u,w} (\beta_{m p^{r-s}})_{w,v} \Bigr)\\
&\= \frac1{m p^r} \bigl(\beta_{mp^r} \- \sum_{s=1}^{r} \gamma_s \cdot \beta_{m p^{r-s}} \bigr)_{u,v} \; \in \; R_{(p)}\\
\eal\] 
by Lemma~\ref{mat_cong_lemma} (iii). Since $h (\tau) \in (R_{(p)} \lb \tau \rb )^g$   and the Jacobian matrix of $h(\tau)$ is the identity matrix, Hazewinkel's functional equation lemma~\cite[\S 10.2(i)]{Ha78} implies that $G_f(\tau,\tau') = l^{-1}(l(\tau)+l(\tau')) \in R_{(p)}\lb \tau, \tau'\rb$.

In the case $J = \Delta(f) \cap \Z^N$ the proof is exactly the same but using bigger matrices $\{ \widetilde \gamma_s ; s \ge 1 \}$, see Remark~\ref{bigger_mat_remark}. 
\end{proof}

\section{Some evidence for the conjecture}\label{sec:ASD}

In this section we assume that $R=\Z_q$, $q=p^k$ is the ring of integers of the unramified extension of $\Q_p$ of degree $k$. We then have $R/pR=\F_q$ and the Frobenius endomorphism $\sigma: R \to R$ satisfies $\sigma^k=Id$. We also assume that the polynomial $f \in R[x_1,\ldots,x_n]$ is homogeneous of degree $d \ge n$ and the hypersurface $X_f = \{f(x)=0\} \subset \P^{n-1}$ is non-singular. By~\cite[Theorem~1]{St87} the formal group law constructed in Theorem~\ref{FGL_theorem} using matrices $\{\beta_m; m \ge 0\}$ with $J=\Delta^\circ \cap \Z^n$ is a coordinalization of the Artin--Mazur formal group $H^{n-2}(X_f,\hat{\mathbb{G}}_{m,X_f})$. Using the relation between Artin--Mazur functors and crystalline cohomology Stienstra proved in~\cite{St87L} the following generalized version of the Atkin and Swinnerton-Dyer congruences. Suppose that the reduction $X_{\bar f}$ is a non-singular hypersurface over $\F_q$ and consider the reciprocal characteristic polynomial of the $q$th power Frobenius operator on the middle crystalline cohomology of $X_{\bar f}$
\be{Frob_poly}
\det (1 \- T \cdot \Fr_q \;|\; H^{N-2}_{crys}(X_{\bar f})) \= 1 \+ c_1 T \+ \ldots \+ c_r T^r \;\in\; \Z[T]\,. 
\ee
By~\cite[Theorem 0.1 and Remark 0.5]{St87L} there exists a constant $c$ such that  
\be{ASD}
\beta_{m} \+ c_1 \, \beta_{m/q} \+ c_2 \, \beta_{m/q^2} \+ \ldots \+ c_r \, \beta_{m/q^r} \equiv 0 \mod \; p^{\ord_p(m) - c}
\ee
whenever $\ord_p(m)$ is sufficiently large. Recall our notation $\alpha_s=\beta_{p^s}$. Combined with Theorem~\ref{alpha_congs}~(ii), congruences~\eqref{ASD} yield that the matrix
\[
\Phi \= F_{\sigma} \cdot \sigma(F_{\sigma}) \cdot \ldots \cdot \sigma^{k-1}(F_{\sigma})  \= \lim_{s \to \infty} \alpha_{s} \cdot \alpha_{s-k}^{-1}
\]
satisfies the equation 
\[
\Phi^r \+ c_1 \Phi^{r-1} \+ \ldots \+ c_r \= 0\,.
\]
It follows that the reciprocal characteristic polynomial of $\Phi$ divides~\eqref{Frob_poly} as we claimed in Section~\ref{sec:intro}, see~\eqref{zeta-factor-lift}. Note that under the conditions of Theorem~\ref{alpha_congs}~(ii) $\det(\Phi)$ is a $p$-adic unit. Therefore the reciprocal eigenvalues of $\Phi$ are $p$-adic unit eigenvalues of the Frobenius operator on the middle crystalline cohomology of the hypersurface $X_{\bar f}$.   

By Katz's congruence formula~\cite{Katz73} the number of $p$-adic unit Frobenius eigenvalues on the middle cohomology equals to the stable rank of the Hasse--Witt matrix. The Hasse--Witt matrix is invertible for a generic non-singular hypersurface (see~\cite{Mill76},\cite{Kob75},\cite{AS16}) and therefore the number of $p$-adic unit Frobenius eigenvalues is equal to the size of the matrix $\Phi$. The conjecture in Section~\ref{sec:intro} would imply that the multiplicities of eigenvalues should also coincide. 

\bigskip

Let us finish with a numerical example. Consider the genus two hyperelliptic curve $C$ given by \[
y^2 \= x^5 \+ 2 x^2 \+ x \+ 1\,.
\]
The two internal integral points in the Newton polytope of this equation are $(1,1)$ and $(2,1)$ and we have
\[
\beta_m \= \text{ the coefficients of } \begin{pmatrix} x^{m-1} y^{m-1} & x^{2 m-1} y^{m-1} \\ x^{m-2} y^{m-1} & x^{2m-2} y^{m-1}\end{pmatrix} \text{ in } (y^2 - x^5 - 2 x^2 - x - 1)^{m-1} \,.
\]
For example, with $p=11$ the higher Hasse--Witt matrices $\alpha_s=\beta_{11^s}$ satisfy  

\begin{center}
\begin{tabular}{c|ccc}
s & 1 & 2 & 3 \\
\hline
${\rm tr}(\alpha_{s} \cdot \alpha_{s-1}^{-1}) \mod 11^s$ &  $8 + O(11)$& $8 + 11 + O(11^2)$ & $8 + 11 + 11^2 + O(11^3)$ \\
${\rm det}(\alpha_{s} \cdot \alpha_{s-1}^{-1}) \mod 11^s$ &  $7 + O(11)$& $7 + 6 \cdot 11 + O(11^2)$ & $7 + 6 \cdot 11 + 3 \cdot 11^2 + O(11^3)$ \\
\end{tabular}
\end{center}
\bigskip

\noindent Using Kedlaya's algorithm we computed the reciprocal characteristic polynomial of the Frobenius on the first crystalline cohomology of the curve reduced modulo $11$:
\[\bal
\det \Bigl(1 - T \cdot \Fr_{11} \;|\; H^1_{crys}(\bar C_{11}) \Bigr) & \= 1 \+ 3 T \+ 18 T^2 \+ 33 T^3 \+ 11^2 T^4  \\
& \= (1 \+ 4 T \+ 11 T^2) (1 \- T \+ 11 T^2)\,.
\eal\]
The eigenvalues of the Frobenius operator are
\[
\lambda_{1,2} \= -2 \pm \sqrt{-7}\,, \qquad \lambda_{3,4} \= \frac{1 \pm \sqrt{-43}}2 \,.
\]
Both $-7$ and $-43$ are squares modulo $11$, and $11$-adic unit eigenvalues are
\[\bal
\lambda_1 &\= 7 \+ 2 \cdot 11 \+ 2 \cdot 11^2 \+ O(11^3)\\
\lambda_3 &\= 1 \+ 10 \cdot 11 \+ 9 \cdot 11^2 \+ O(11^3)\\
\eal\]
We see that in the above table traces and determinants converge to
\[\bal
\lambda_1 \+ \lambda_3 &\= 8 + 11 + 11^2 + O(11^3)\\
\lambda_1 \cdot \lambda_3 &\= 7 + 6 \cdot 11 + 3 \cdot 11^2 + O(11^3)\\
\eal\]
respectively.

\end{document}